\documentclass[]{amsart}
\usepackage{amscd,amsthm,amssymb,amsfonts,amsmath,euscript}

\theoremstyle{plain}
\newtheorem{thm}{Theorem}[section]
\newtheorem{lemma}[thm]{Lemma}
\newtheorem{prop}[thm]{Proposition}
\newtheorem{cor}[thm]{Corollary}

\theoremstyle{definition}
\newtheorem{defn}[thm]{Definition}

\theoremstyle{remark}
\newtheorem{remark}[thm]{Remark}

\newcommand{\nc}{\newcommand}

\def\makeop#1{\expandafter\def\csname#1\endcsname
  {\mathop{\rm #1}\nolimits}\ignorespaces}
\makeop{Hom}   \makeop{End}   \makeop{Aut}   \makeop{Isom}  \makeop{Pic} 
\makeop{Gal}   \makeop{ord}   \makeop{Char}  \makeop{Div}   \makeop{Lie} 
\makeop{PGL}   \makeop{Corr}  \makeop{PSL}   \makeop{sgn}   \makeop{Spf}
\makeop{Spec}  \makeop{Tr}    \makeop{Nr}    \makeop{Fr}    \makeop{disc}
\makeop{Proj}  \makeop{supp}  \makeop{ker}   \makeop{im}    \makeop{dom}
\makeop{coker} \makeop{Stab}  \makeop{SO}    \makeop{SL}    \makeop{SL}
\makeop{Cl}    \makeop{cond}  \makeop{Br}    \makeop{inv}   \makeop{rank}
\makeop{id}    \makeop{Fil}   \makeop{Frac}  \makeop{GL}    \makeop{SU}
\makeop{Nrd}   \makeop{Sp}    \makeop{Tr}    \makeop{Trd}   \makeop{diag}
\makeop{Res}   \makeop{ind}   \makeop{depth} \makeop{Tr}    \makeop{st}
\makeop{Ad}    \makeop{Int}   \makeop{tr}    \makeop{Sym}   \makeop{can}
\makeop{length}\makeop{SO}    \makeop{torsion} \makeop{GSp} \makeop{Ker}
\makeop{Adm}   \makeop{Mat}
\def\makebb#1{\expandafter\def
  \csname bb#1\endcsname{{\mathbb{#1}}}\ignorespaces}
\def\makebf#1{\expandafter\def\csname bf#1\endcsname{{\bf
      #1}}\ignorespaces} 
\def\makegr#1{\expandafter\def
  \csname gr#1\endcsname{{\mathfrak{#1}}}\ignorespaces}
\def\makescr#1{\expandafter\def
  \csname scr#1\endcsname{{\EuScript{#1}}}\ignorespaces}
\def\makecal#1{\expandafter\def\csname cal#1\endcsname{{\mathcal
      #1}}\ignorespaces} 

\def\doLetters#1{#1A #1B #1C #1D #1E #1F #1G #1H #1I #1J #1K #1L #1M
                 #1N #1O #1P #1Q #1R #1S #1T #1U #1V #1W #1X #1Y #1Z}
\def\doletters#1{#1a #1b #1c #1d #1e #1f #1g #1h #1i #1j #1k #1l #1m
                 #1n #1o #1p #1q #1r #1s #1t #1u #1v #1w #1x #1y #1z}
\doLetters\makebb   \doLetters\makecal  \doLetters\makebf
\doLetters\makescr 
\doletters\makebf   \doLetters\makegr   \doletters\makegr
     \def\qed{\qedmark\medbreak}%
\def\qedmark{{\enspace\vrule height 6pt width 5pt depth 1.5pt}}%
    
\def\Gm{{{\bbG}_{\rm m}}}   

\normalsize

\makeop{Bl}

\def\Spec{{\rm Spec}\,}

\def\Fpbar{\overline{\bbF}_p}
\def\Fp{{\bbF}_p}
\def\Fq{{\bbF}_q}

\def\Qp{{\bbQ}_p}

\def\Zp{{\bbZ}_p}

\newcommand{\Z}{\mathbb Z}
\newcommand{\Q}{\mathbb Q}

\newcommand{\D}{\mathbf D}    
\newcommand{\F}{\mathbb F}

\newcommand{\npr}{\noindent }

\newcommand{\<}{\langle}   
\renewcommand{\>}{\rangle} 

\nc{\embed}{\hookrightarrow}

\newcommand{\ch}{characteristic }
\newcommand{\ac}{algebraically closed }
\newcommand{\dieu}{Dieudonn\'{e} }

\nc{\ol}{\overline}
\nc{\wt}{\widetilde}
\nc{\opp}{\mathrm{opp}}
\def\ul{\underline}

\makeop{Ram}
\makeop{Rep}

\begin{document}
\renewcommand{\thefootnote}{\fnsymbol{footnote}}
\setcounter{footnote}{-1}
\numberwithin{equation}{section}


\title[Basic abelian varieties]
{Abelian varieties over finite fields as basic abelian varieties}
\author{Chia-Fu Yu}
\address{
Institute of Mathematics, Academia Sinica and NCTS (Taipei Office)\\
Astronomy Mathematics Building \\
No. 1, Roosevelt Rd. Sec. 4 \\ 
Taipei, Taiwan, 10617} 
\email{chiafu@math.sinica.edu.tw}
\address{
  The Max-Planck-Institut f\"ur Mathematik \\
  Vivatsgasse 7, Bonn \\
  Germany 53111} \email{chiafu@mpim-bonn.mpg.de}

\date{\today}
\subjclass[2010]{} 


\begin{abstract}
In this note we show that any basic abelian variety with additional
structures over an arbitrary \ac field of \ch $p>0$ is isogenous to
another one defined over a finite field. We also show
that the category of abelian varieties over finite fields up to
isogeny can be embedded into the category of basic abelian varieties
with suitable endomorphism structures. Using this connection, we
derive a new mass formula for a finite orbit 
of polarized abelian surfaces over a finite field.   
\end{abstract} 

\maketitle


\section{Introduction}
\label{sec:intro}
\def\Mass{{\rm Mass}}
In this note we work on abelian varieties over fields of
characteristic $p>0$, 
particularly on basic abelian varieties with additional
structures (endomorphisms, a polarization and a level
structure).
Conceptually, an abelian variety with fixed
additional structures is \emph{basic} if the corresponding point in 
a moduli space of PEL-type over $\Fpbar$ lands in the minimal 
Newton stratum 
(Rapoport-Zink~\cite{rapoport-zink} and Rapoport~\cite{rapoport:np}). 
The group-theoretic definition was introduced by Kottwitz
\cite{kottwitz:isocrystals}.  
This notion is geometric
in the sense that an abelian variety with additional structures is basic
if and only if its base change to any algebraically closed field
extension is also basic. 
As isogenous abelian varieties land in the same Newton stratum,
an abelian variety with additional structures that is
isogenous to a basic one is also basic. 

Let $B$ be a finite-dimensional semi-simple $\Q$-algebra with a
positive involution $*$ and $O_B$ an order in $B$ stable under $*$. 
A polarized $O_B$-abelian variety is a triple $(A,\lambda,\iota)$
where $A$ is an abelian variety with polarization $\lambda$ and $\iota:
O_B\to \End(A)$ is a ring monomorphism 
which is compatible with $\lambda$. 
We recall the definition of basic polarized $O_B$-abelian varieties
$(A,\lambda,\iota)$ in Section~\ref{sec:02}.

Basic abelian varieties with additional structures share many
similar properties with supersingular abelian varieties without 
additional structures, and many techniques employed there can be
carried over here as well. 
For example, similar to supersingular abelian
varieties, one can formulate a geometric mass for a finite orbit of
basic abelian varieties and relate this geometric mass to an
arithmetic mass defined by group theory.
The well-known Deuring-Eichler mass formula is obtained in this fashion.  
We refer to \cite{yu:smf} for more discussions in
this aspect. In this paper we prove the following result, which may
be regarded as another analogue property enjoyed 
by supersingular abelian varieties. 

\begin{thm}\label{11}
  Let $\ul A=(A,\lambda,\iota)$ be a basic polarized $O_B$-abelian
  variety over 
  an \ac field $k$ of \ch $p>0$. 
  Then there exists a polarized $O_B$-abelian variety $\ul
  A'=(A',\lambda',\iota')$ over a
  finite field $\kappa$ and an $O_B$-linear isogeny
  $\varphi:A'\otimes_\kappa k\to A$ over $k$ that preserves 
  the polarizations. 
\end{thm}



The second part of this note studies the converse to 
Theorem~\ref{11}. 
We show that any abelian variety over a finite field can be regarded
as a basic abelian variety with suitable endomorphism structures. More
precisely, if $A$ is an abelian variety over the finite field $\Fq$ of
$q=p^s$ elements and 
$F=\Q(\pi_A)\subset \End(A)\otimes \Q$ is the $\Q$-subalgebra 
generated by its Frobenius
endomorphism $\pi_A$, then the abelian variety $A$ together 
with the $F$-action is a basic $F$-abelian
variety (Proposition~\ref{41}). See Remark~\ref{rem:32} for 
the notion of a $B$-abelian variety being basic.
A priori, the original definition of
basic abelian varieties with  
additional structures requires both structures of 
endomorphisms and polarizations. 
However, similar to supersingular abelian varieties, 
polarizations play no role in the characterization 
of supersingularity.

Let $\calA_{\Fq}$ denote the category of abelian varieties over $\Fq$
up to isogeny, and $\calB^{\rm rig}$ be the category of basic
abelian varieties with rigidified endomorphisms over
$\Fpbar$ up to isogeny, defined in Section~\ref{sec:04}. 
We prove the following result.

\begin{thm}\label{12}
  There is a functor $\Phi$ that embeds the category $\calA_{\Fq}$
  as a full subcategory of $\calB^{\rm rig}$. 
\end{thm}

Theorem~\ref{12} connects (polarized) abelian varieties
over a finite field $\Fq$ with basic (polarized) $F$-abelian 
varieties over $\Fpbar$ equipped with  a suitable commutative 
semi-simple $\Q$-algebra $F$. 
This connection is particularly useful when the $\Q$-algebra $F$ is fixed. 
In this case one may consider a smaller class of (polarized) 
abelian varieties over $\Fq$ 
whose endomorphism rings
contain the maximal order $O_F$. 
This smaller set of isomorphism classes 
of polarized abelian varieties over $\Fq$ 
is embeddable into the basic locus of a moduli
space of polarized $O_F$-abelian varieties; see Lemma~\ref{51} and
(\ref{eq:52}). Below is a example where we use this
embedding to derive a mass formula for a class of polarized abelian
surfaces over $\Fp$. 

Choose a simple abelian variety $A_0$ over the prime finite field 
$\Fp$ whose Frobenius endomorphism $\pi_0$ satisfies that
$\pi_0^2=p$. Then $A_0$ is a superspecial abelian surface, i.e. the base
change $A_0\otimes \Fpbar$ is isomorphic to the product of two
supersingular elliptic curves.   
Let us consider the set $\Lambda$ 
of isomorphism classes of  
principally polarized simple abelian surfaces $(A,\lambda)$ over
$\Fp$ such that $A$ is isogenous to $A_0$. 
Put $F=\Q(\pi_0)=\Q(\sqrt{p})$ and
$O_F$ its ring of integers. Let $\Lambda^{\rm max}\subset \Lambda$ be the
subset of classes $[(A,\lambda)]$ such that $O_F\subset \End(A)$.
We can show that $\Lambda^{\rm max}$ is a nonempty set. 
As usual, the mass $\Mass(\Lambda^{\rm max})$ of $\Lambda^{\rm max}$ 
is defined by
\begin{equation}
  \label{eq:11}
  \Mass(\Lambda^{\rm max}):=\sum_{(A,\lambda)\in \Lambda^{\rm max}}
|\Aut(A,\lambda)|^{-1}.
\end{equation}
Then we show that $\Mass(\Lambda^{\rm max})$ 
is equal to the mass of a finite Hecke orbit $S$ 
in the superspecial
locus of a Hilbert modular surface modulo $p$. Furthermore, 
using the geometric
mass formula for the superspecial orbits established in 
\cite{yu:mass_hb}, we obtain the mass formula
\begin{equation}
  \label{eq:12}
  \Mass(\Lambda^{\rm max})=\frac{\zeta_{F}(-1)}{4},
\end{equation}
where $\zeta_F(s)$ the Dedekind zeta
function of $F$ (see Section~\ref{sec:52}). 

The paper is organized as follows. In Section~\ref{sec:02} we recall
the definition of basic abelian varieties with additional structures.
The proof and some consequences of Theorem~\ref{11} are given 
in Section~\ref{sec:03}. In Section~\ref{sec:04} we show that any
abelian variety over a finite field, together with the action of the
center of its endomorphism algebra, is a basic abelian variety. This
result is used to construct the functor $\Phi$ in Theorem~\ref{12}.
In the last section we
consider the isogeny class of simple supersingular abelian surfaces 
mentioned as above and compute the
associated mass (\ref{eq:11}). \\



\npr {\bf Notation.} If $M$ is a $\Z$-module or a $\Q$-module and
$\ell$ is a prime, 
we write $M_\ell:=M\otimes_{\Z} \Z_\ell$ or $M_\ell=M\otimes_{\Q} \Q_\ell$,
respectively.  
For any perfect field $k$ of \ch $p>0$,
denote by $W(k)$ the ring of Witt vectors over $k$, $B(k)$ the
field of fractions of $W(k)$, $\sigma$ the Frobenius map on $W(k)$ and
$B(k)$ induced by $\sigma:k\to k,\ x\mapsto x^p$. If $F$ is a finite
product of number fields $F_i$, denote by $O_F$ the maximal order in
$F$. A prime $\bfp$ of $F$ over $p$, denoted by $\bfp|p$,
means a prime of $F_i$ for some $F_i$ or a prime ideal of $O_F$ over $p$.
For an abelian variety $A$ over a field $k$, write $\End(A)=\End_k(A)$
for the endomorphism ring of $A$ over $k$ and
$\End^0(A)=\End(A)\otimes_\Z \Q$ for the endomorphism algebra of $A$
over $k$. If $A$ is defined over a finite field $\Fq$, we denote by
$\pi_A$ the Frobenius endomorphism of $A$ over $\Fq$.  

\section{Basic abelian varieties with additional structures}
\label{sec:02}

In this section we recall the concept of 
basic abelian varieties with additional
structures introduced by Kottwitz \cite{kottwitz:isocrystals}. 
Our reference is  
Rapoport-Zink \cite[p.11, p.~281 and 6.25, p.~291]{rapoport-zink}.

\subsection{Setting}
\label{sec:21}
\def\GAut{{\rm GAut}}
\def\char{{\rm char}}
\def\GU{{\rm GU}}

Let $B$ be a finite-dimensional semi-simple algebra over $\Q$ with a
positive involution $*$, and $O_B$ be an arbitrary order of $B$ 
stable under $*$. 

Recall that a non-degenerate
$\Q$-valued skew-Hermitian $B$-space  is a pair $(V,\psi)$ where 
$V$ is a left faithful finite $B$-module, and 
$\psi:V\times V\to \Q$ is a non-degenerate alternating pairing
such that $\psi(bx,y)=\psi(x,b^* y)$ for all $b\in B$ and 
all $x,y\in V$.

A {\it polarized $O_B$-abelian variety} (resp.~{\it polarized 
$B$-abelian variety}) is a
triple $\ul A=(A,\lambda,\iota)$, where $(A,\lambda)$ is a polarized
abelian 
variety and $\iota:O_B\to \End(A)$ (resp. $\iota:B\to
  \End^0(A)$) 
is a ring monomorphism such that
$\lambda \iota(b^*)=\iota(b)^t \lambda$ for all $b\in O_B$. Here 
$\iota(b)^t:A^t\to A^t$ denotes the dual morphism of $\iota(b)$.

Let $\ul A$ be a polarized $O_B$-abelian variety over $k$, where $k$
is an arbitrary field. 
For any prime $\ell$ (not necessarily invertible in $k$), 
we write $\ul A(\ell)$ for the associated
$\ell$-divisible group with additional structures $(A[\ell^\infty],
\lambda_\ell, \iota_\ell)$, where $\lambda_\ell$ is the induced
quasi-polarization from $A[\ell^\infty]$ to
$A^t[\ell^\infty]=A[\ell^\infty]^t$ (the Serre dual), 
and $\iota_\ell:(O_B)_\ell \to \End(A[\ell^\infty])$
the induced ring monomorphism. If $\ell\neq \char (k)$, let
$T_\ell(A)$ denote the $\ell$-adic Tate module of $A$,  
$V_\ell:=T_\ell(A)\otimes \Q_\ell$, and 
let
\begin{equation}
  \label{eq:21}
  \rho_\ell: \calG_k\to \GU_{B_\ell}(V_\ell,e_{\lambda})
\end{equation}
be the associated Galois representation. Here 
$\calG_k:=\Gal(k_s/k)$ is the Galois group of $k$, 
$k_s$ a separably closure of $k$, 
and 
\[ \GU_{B_\ell}(V_\ell,e_{\lambda}):=\{g\in \Aut_{B_\ell}(V_\ell)\,|\,
e_{\lambda}(gx,gy)=c\, e_{\lambda} (x,y)\ 
\text{for some $c\in \Q_\ell^\times$}\,
  \} \]
is the group of $B_\ell$-linear similitudes with respect to the 
Weil pairing 
\[ e_{\lambda}=e_{\lambda,\ell}: T_\ell(A)\times T_\ell(A) \to
\Z_\ell(1), \] 
where 
\[\Z_\ell(1) :=\lim_{\leftarrow} \mu_{\ell^m}(k_s) \] 
is the Tate twist. 

If $k$ is a perfect field of \ch $p$, let $M(\ul A)$ denote the
covariant \dieu module of $\ul A$ with the additional structures and
put $N(\ul A):=M(\ul A)\otimes_{W(k)} B(k)$, the rational \dieu module
(or the isocrystal) with the additional structures; 
see \cite[Section 1]{yu:reduction}.


In this note we consider only the objects $\ul A=(A,\lambda,\iota)$ 
for which there is a non-degenerate skew-Hermitian $B$-space $(V,\psi)$
with $2 \dim A = \dim_\Q V$. Namely, we require that there exists a 
complex polarized $O_B$-abelian variety with the same dimension as $A$. 
For example, we exclude the
case where $A$ is a supersingular elliptic curve and $B$ is the
quaternion $\Q$-algebra ramified precisely at $\{p,\infty\}$.   

\subsection{Basic abelian varieties}

Let $k$ be any field of \ch $p$ and $\bar k$ be an algebraic
closure of $k$. Put $W:=W(\bar k)$ and $L:=B(\bar k)$. 
Let $(V_p,\psi_p)$ be a $\Q_p$-valued non-degenerate skew-Hermitian
  $B_p$-module. 
  A polarized $O_B$-abelian
  variety $\ul A$ over $\bar k$ is said
  to be {\it related to $(V_p,\psi_p)$} if there is a 
  $(B_p\otimes_{\Qp} L)$-linear isomorphism $\alpha:N(\ul A) \simeq
  (V_p,\psi_p)\otimes_{\Qp} L$ which preserves the pairings for 
  a suitable identification $L(1)\simeq L$.

  Let $G_p:=\GU_{B_p}(V_p,\psi_p)$ be the algebraic
  group over $\Q_p$ of $B_p$-linear similitudes with respect to the
  pairing $\psi_p$. 
  A choice of $\alpha$ gives rise to
  an element $b\in G_p(L)$ by transport of structure of the Frobenius
  map on $N(\ul A)$, that is, 
  $\alpha: N(\ul A) \simeq (V_p\otimes L, b({\rm id}\otimes \sigma),
  \psi_p)$ becomes an isomorphism of
  isocrystals with additional structures. The $\sigma$-conjugacy class
  $[b]$ of $b$ in
  $G_p(L)$ is independent of the choice of $\alpha$. 
  The decomposition of $V_p\otimes L$ into
  isotypic components (the components of a single slope) 
  induces a $\Q$-graded structure, and thus
  defines a (slope) homomorphism $\nu_{b}:\D\to G_p$ over some
  unramified finite
  extension $\Q_{p^s}$ of $\Q_p$, where $\D$ is the pro-torus over
  $\Q_p$ with character group $\Q$. The set $\nu_{[b]}=\{\nu_b\}$ for
  $b\in [b]$ 
  is the $G_p(L)$-conjugacy class of $\nu_b$ for a single $b\in [b]$,
  called the Newton vector associated to $N(\ul A)$.     

\begin{defn} \label{def:basic}

{\rm (1)} A polarized $O_B$-abelian variety $\ul A$ over $\bar k$ 
  is said to be {\it
  basic with respect to $(V_p,\psi_p)$} if 
\begin{enumerate}
\item [(a)] $\ul A$ is related to $(V_p,\psi_p)$, and
\item [(b)] the slope homomorphism $\nu_{b}: \D\to G_p$ for $b\in [b]$
  is central.     
\end{enumerate}

{\rm (2)} The object $\ul A$ over $\bar k$ is said to be 
  {\it basic} if it is basic
  with respect to $(V_p,\psi_p)$ for some non-degenerate
  skew-Hermitian $B_p$-space $(V_p,\psi_p)$.  

{\rm (3)}  A polarized $O_B$-abelian variety $\ul A$ over any field
  $k$ is said to be {\it basic} if its base change $\ul A\otimes_k
  \bar k$ is basic.
\end{defn}





Clearly a polarized $O_B$-abelian variety $\ul A$ is basic if (and only
if) it is so considered as polarized $B$-abelian variety. 
Two polarized $B$-abelian varieties $\ul A_1$ and $\ul A_2$ are
said to be {\it isogenous}, denote $\ul A_1 \sim \ul A_2$,  
if there is a $B$-linear isogeny $\varphi: A_1\to
A_2$ such that the pull-back $\varphi^* \lambda_2$ is a $\Q$-multiple of
$\lambda_1$. Clearly the property for an object $\ul A$ being basic 
is an isogeny invariant property. 
From the definition it is also easy to see that this is a geometric
notion: an object $\ul A=(A,\lambda,\iota)$
over $k$ is basic if and only if the base change 
$\ul A \otimes_k k_1$ is basic for any \ac field $k_1 \supset k$. 

\section{Proof of Theorems~\ref{11} and its corollaries}
\label{sec:03}

\subsection{}\label{sec:31} 
To prove Theorem~\ref{11}, we need some properties of basic abelian
  varieties with additional structures. 
Let $(V,\psi)$ be a  
non-degenerate ($\Q$-valued) skew-Hermitian $B$-space and let
$G:=\GU_B(V,\psi)$ be the algebraic group over $\Q$ of $B$-linear
similitudes with respect to the pairing $\psi$. 

Let $F$ be the center of $B$ and
$F_0$ be the $\Q$-subalgebra fixed by the induced 
involution on $F$, which we 
denote by $a\mapsto \bar a$. 
Let $\Sigma_p$ be the set of primes  
$\bfp$ of $F$ over $p$, and for each prime $\bfp|p$, 
denote by $\ord_\bfp$ the
corresponding $p$-adic valuation normalized in a way 
that $\ord_\bfp(p)=1$. 
Let $F_p:=F\otimes \Q_p=\prod_{\bfp|p} F_\bfp$ be the decomposition 
into a product of local fields. For each isocrystal $N$ with an
$F_p$-linear action, let 
\begin{equation}
  \label{eq:N}
  N=\oplus_{\bfp|p} N_\bfp
\end{equation}
be the decomposition with respect to the $F_p$-action.

\begin{lemma}[Rapoport-Zink]\label{31} Let the notation be as above. 

{\rm (1)} The center $Z$ of $G$ is the algebraic subgroup over $\Q$
      whose 
      group of $R$-points is 
    \[ Z(R)=\{x\,\in (F\otimes R)^\times;
    \, x \bar x\in R^\times \,\},  \]
   for any $\Q$-algebra $R$.

{\rm (2)} Let $N$ be an isocrystal with additional structures and
suppose that it is related to $(V\otimes \Q_p,\psi)$. 
Then $N$ is basic with respect to  $(V\otimes \Q_p,\psi)$ 
if and only if each component
$N_\bfp$ is isotypic. In particular, if $N$ is basic, then $N_\bfp$ is
supersingular for primes $\bfp$ with $\bfp=\bar \bfp$. 
\end{lemma}
\begin{proof}
  Statement (1) and the only if part of statement (2) 
  are proved in 6.25 of
  \cite{rapoport-zink}. The if part is easier: as each $N_\bfp$ is
  isotypic, say of slope $r_\bfp/s$, the slope homomorphism $s\nu_b$
  factors through $\bfD\to \Gm$ and the action of $s \nu_b(p)$ on
  $N_\bfp$
  is a scalar. Thus, the slope homomorphism $\nu_b:\bfD\to G_p$ 
  must be central. \qed 
\end{proof}


\begin{remark}\label{rem:32}
Lemma~\ref{31} provides a simple criterion for checking a
polarized $B$-abelian variety $\ul A=(A,\lambda,\iota)$ being basic. 
Note that the assertion of the statement (2)  
depends only on the underlying structure of
$B$-action, and not on the equipped polarization structure. 
Therefore, it makes sense to call a $B$-abelian variety 
$(A,\iota)$ {\it basic} 
if for any $B$-linear polarization $\lambda$ on $(A,\iota)$, 
the polarized $B$-abelian
variety $(A,\lambda,\iota)$ is basic in the sense of 
Definition~\ref{def:basic}. 
Such a polarization $\lambda$ always exists; see Kottwitz \cite[Lemma
9.2]{kottwitz:jams92}.
 

It follows from Lemma~\ref{31} that a $B$-abelian variety $(A,\iota)$
is basic if and only if the $F$-abelian variety $(A,\iota|_F)$ is
basic, where 
$\iota|_F$ is the
restriction of $\iota$ to $F$.         
\end{remark}

The following two lemmas are reorganized from
\cite[6.26-6.29]{rapoport-zink}; proofs are provided solely for the
reader's convenience. 

\begin{lemma}\label{32}
  Given any set $\{\lambda_\bfp\}_{\bfp|p}$ of rational numbers with
  $0\le \lambda_{\bfp}\le 1$ and $\lambda_{\bfp}+\lambda_{\bar
  \bfp}=1$, there is a positive integer $s$ and $u\in
  O_F[1/p]^\times$ such that 
  \begin{equation}
    \label{eq:u}
   u \ol u=q:=p^s,\quad\text{and }\ \ord_{\bfp} u=s \lambda_\bfp, \
  \forall\, \bfp\in 
  \Sigma_p.  
  \end{equation}
\end{lemma}
\begin{proof}
  Consider the map 
  \[ \ord:O_F \left [\frac{1}{p} \right ]^\times\to \bigoplus_{\bfp\in
    \Sigma_p} (1/e_\bfp)\Z,
  \quad  
  u \mapsto (\ord_{\bfp}(u))_{\bfp\in \Sigma_p}, \] 
  where $e_\bfp$ is the ramification index of $\bfp$.
  By Dirichlet's unit theorem, the image has rank $|\Sigma_p|$ and 
  is of finite index. Therefore, there are a positive integer $s$ 
  and an element $u\in O_F[1/p]^\times$ such that
  $\ord_\bfp(u)=s\lambda_\bfp=:r_\bfp$ for all $\bfp\in \Sigma_p$. Let
  $q=p^s$ and $u':=q u/\bar u$, then one computes 
\[ \ord_\bfp u'=2r_\bfp\  \text{and}\  u' \ol u'=q^2. \]
Replacing $u$ by $u'$ and $q$ by $q^2$, one gets the 
desired result. \qed   
\end{proof}

The element $u$ in Lemma~\ref{32} actually lies in $O_F$ 
as $\ord_\bfp(u)\ge 0$ for all $\bfp|p$. 

\begin{lemma}\label{33}
  Fix $\{\lambda_\bfp\}_{\bfp|p}$ and $q=p^s$ as in
  Lemma~\ref{32}, and an positive integer $g$. Then 
  there is a positive integer $n$ such that for any basic
  $g$-dimensional polarized
  $O_B$-abelian variety $\ul A$ over a finite extension $\F_{q^m}$ of
  $\F_q$ with slopes $\{\lambda_\bfp\}_{\bfp|p}$, the $n$-th power of
  Frobenius morphism $\pi_A^n$ lies in $\iota(F)$.
\end{lemma}
\begin{proof}
  We first prove that the statement holds for one such object $\ul
  A$, i.e. there is an integer $n_A$ possibly depending on $A$ 
  such that $\pi_A^{n_A}\in \iota(F)$. Clearly the statement 
  depends only on the isogeny class of $\ul A$. 
  Let $M$ be the \dieu module of $\ul A$. Within the isogeny
  class, we 
  can choose $\ul A$ so that $\iota(O_F)\subset \End(A)$ 
  and $F^s M_\bfp=p^{r_\bfp} M_\bfp$ for all
  $\bfp\in \Sigma_p$, where $r_\bfp=s\lambda_\bfp$ and
  $M=\oplus_{\bfp|p} M_\bfp$ is the decomposition 
  with respect to (\ref{eq:N}). 
  Let $u$ be as in Lemma~\ref{32}, then
  $\iota(u)^{-m}\pi_A$ is an automorphism of $A$ that preserves the
  polarization as $\iota(u)^{-m}\pi_A(M_\bfp)=M_\bfp$ for all
  $\bfp\in \Sigma_p$. Therefore, a power of this automorphism 
  is the identity by a theorem of Serre. Thus, a power of $\pi_A$ 
  is contained in $\iota(F)$.

  Let $C:=\End^0_B(A)$. As $\dim C$ is bounded by $4g^2$, there is a
  fixed positive integer $n$ such that $\zeta^n=1$ for 
  any element $\zeta \in C$ of finite order. By the result we just
  proved that $\iota(u)^{-m}\pi_A\in C$ is of finite order,  
  we have $\pi_A^n\in \iota(F)$ for all such objects $\ul A$. \qed

   
\end{proof}

\subsection{Proof of Theorem~\ref{11}}
\label{sec:3.2}
  Let the notation be as in Theorem~\ref{11}.
  It suffices to show that $A$ has smCM, that is, any maximal 
  commutative semi-simple $\Q$-subalgebra of $\End^0(A)$ has degree  
  $2 \dim A$. Then by a theorem of Grothendieck (see a proof in
  \cite{oort:cm} or \cite{yu:cm}) there exists an
  abelian variety $A'$ over a finite field $\kappa$ and an
  isogeny $\varphi:A'\otimes_\kappa k \to A$ over $k$. Replacing
  $A'$ by one in its isogeny class if necessary, we may assume
  that $A'$ admits an action $\iota'$ of $O_B$ so that 
  the isogeny $\varphi$ is $O_B$-linear. Take the
  pull-back polarization $\lambda'$ on $A'$, which is clearly 
  defined over a finite field extension of $\kappa$. 

  Let $\{\lambda_\bfp\}_{\bfp|p}$ be the set of slopes for $\ul
  A$. Take $q=p^s$ and a positive integer $n$ as in Lemmas~\ref{32} 
  and \ref{33}. We can choose a 
  field $k_0$ finitely generated over $\F_q$ over which $\ul A$ is
  defined. The abelian variety
  $\ul A$ extends to a polarized $O_B$-abelian scheme $\ul {\bf A}$
  over $S=\Spec R$ for a finitely generated $\Fq$-subalgebra $R$ of $k_0$
  with fraction field $\Frac(R)=k_0$. We may assume further that $S$
  is smooth over $\Spec \Fq$. Let $s$ be a closed point
  of $S$ and $\eta$ the generic point. By Grothendieck's
  specialization theorem, the special fiber $\ul {\bf A}_s$ over $s$
  also has the same slopes $\{\lambda_\bfp\}_{\bfp|p}$, and hence is
  basic. 

  We identify the endomorphism rings $\End_{k_0}(A)=\End_R({\bf
  A})\subset \End({\bf A}_{\bar s})$, and write $\iota$ for the
  $O_B$-actions on these abelian varieties. Let
  \[ \rho_\ell:\pi_1(S,\bar \eta)\to \Aut (T_\ell(A_{\bar \eta}))\] be
  the associated $\ell$-adic representation. The action of $\Gal(\bar
  \eta/\eta)$ on $T_\ell(A_{\bar \eta})$ factors through $\rho_\ell$. 
  Again we identify the
  Tate modules $T_\ell({\bf A}_{\bar s})=T_\ell({\bf A}_{\wt S_{\bar
  s}})=T_\ell({A}_{\bar \eta})$, where $\wt S_{\bar s}$ is the (strict)
  Henselization of $S$ at $\bar s$. Put $V_\ell({A}_{\bar \eta}):=
  T_\ell({A}_{\bar \eta})\otimes \Q_\ell$.  

  Let $\pi_{A_s}$ be the Frobenius morphism on ${\bf A}_s$
  and ${\rm Frob}_s$ the geometric Frobenius element in $\pi_1(S,\bar
  \eta)$ corresponding to the closed point $s$. We have 
  \begin{itemize}
  \item [(i)] $\pi^n_{A_s}\in \iota(F)\subset \End (T_\ell({\bf
  A}_{\bar s}))$, by Lemma~\ref{33};
  \item [(ii)] $\rho_{\ell}({\rm Frob}_s^n)= \pi^n_{A_s}$ lies in the
  center $Z(\Q_\ell)$ of 
  $\GU_{B_\ell}(V_\ell(A_{\bar \eta}),\<\, ,\>)$, 
  by the identification of the Tate modules and (i);
  \item [(iii)] the Frobenius elements ${\rm Frob}_s$ for all closed
  points $s$ generate a dense subgroup of $\pi_1(S,\bar \eta)$. 
  \end{itemize} 

  Let $G_\ell:=\rho_\ell(\pi_1(S,\bar \eta))$ be the $\ell$-adic
  monodromy group. Let $m_n:G_\ell\to G_\ell$ be the map $x\mapsto
  x^n$.  
  It is an open mapping and its image contains an open
  subgroup $U$ of $G_\ell$, which is of finite index. 
  Clearly $U$ lies in the center
  $Z(\Q_\ell)$ by (ii) and (iii). Replacing $k_0$ by a finite extension,
  we have $G_\ell\subset Z(\Q_\ell)$. Let $\Q_\ell[\pi]$ be the
  (commutative) subalgebra of $\End (V_\ell(A_{\bar \eta}))$ generated
  by 
  $G_\ell$. By Zarhin's theorem \cite{zarhin:end}, 
  $\Q_\ell[\pi]$ is semi-simple and
  commutative, and $\End_{\Q_\ell[\pi]} (V_\ell(A_{\bar \eta}))=\End
  (A)\otimes \Q_\ell$. Hence any maximal commutative semi-simple
  $\Q_\ell$-subalgebra of $\End(A)\otimes \Q_\ell$ is also a maximal
  one in $\End(V_\ell(A_{\bar \eta}))$. 
  This shows that any maximal commutative
  semi-simple 
  subalgebra of $\End^0(A)$ has degree $2g$, and hence completes the
  proof. \qed



\subsection{Consequences}
\label{sec:34}

In \cite{yu:smf} we defined a class of polarized $B$-abelian 
varieties, called of arithmetic type. For these abelian
varieties the ``simple mass formula'' in 
\cite[Theorem 2.2]{yu:smf} remain valid for \ac ground fields, not
just for finitely generated fields over a prime field. 
We related these $B$-abelian varieties with
basic $B$-abelian varieties in the case where the ground field $k$ is
$\Fpbar$; see \cite[Theorem 4.5]{yu:smf}. Using Theorem~\ref{11},  
we extend this result to
an arbitrary \ac field k of \ch $p>0$,

Recall that a polarized $B$-abelian variety $(A,\lambda, \iota)$
over an \ac field $k$ of \ch $p>0$ is said to 
be {\it of arithmetic type} 
if there is a model $(A_0,\lambda_0,\iota_0)$ of $(A,\lambda,\iota)$ 
over a subfield $k_0$ finitely generated over $\Fp$ such that the 
associated Galois
representation $\rho_\ell: \calG_{k_0}\to \GU_{B}(V_\ell(A_0),
e_{\lambda,\ell})$ (Section~\ref{sec:21}) is central 
for some prime $\ell\neq p$ 
(or equivalently, for all primes $\ell\neq p$, 
see \cite[Proposition 3.10]{yu:smf}). 
It is shown in \cite[Section 3]{yu:smf} that 
this is again a geometric notion which depends only on the
underlying $B$-abelian variety $(A,\iota)$ 
and not on the carried polarization structure $\lambda$.  
\begin{thm}\label{35}
  A $B$-abelian variety $(A,\iota)$ over an \ac field $k$ of \ch
  $p>0$ is of arithmetic type if and only if it is basic. 
\end{thm}
\begin{proof}
  By Theorem~\ref{11}, there is a $B$-abelian variety 
  $(A_0,\iota_0)$ over $\Fpbar$ and a $B$-linear isogeny $\varphi:
  (A_0,\iota_0)\otimes_{\Fpbar} k \to (A,\iota)$. As a result we can
  reduce the statement to the case where $k=\Fpbar$ and this is Theorem
  4.5 of \cite{yu:smf}. \qed
\end{proof}

\begin{prop}[cf.~{\cite[Corollary 6.29]{rapoport-zink}}\,]\label{36}
  Let $K$ be a finite-dimensional 
  semi-simple $\Q$-algebra that admits
  a positive involution. Let $(A,\iota)$ and $(A',\iota')$ be two
  basic $K$-abelian varieties over an \ac field $k$ of \ch $p>0$. Then 
  we have
  \begin{equation}
    \label{eq:V}
   \Hom_{K}(A,A')\otimes_{\Z} \Q_\ell \simeq \Hom_{K}
  (V_\ell(A),V_\ell(A')) \quad \forall\, \ell \neq p, 
  \end{equation}
and 
\begin{equation}
  \label{eq:NF}
\Hom_{K}(A,A')\otimes_{\Z} \Q_p \simeq \Hom_{K}
  ((N,\calF),(N',\calF)),
\end{equation}
where $N$ and $N'$ are 
the isocrystals associated to $(A,\iota)$ and $(A',\iota')$, respectively. 
\end{prop}
\begin{proof}
  Let $E$ be the center of $K$. If (\ref{eq:V}) and (\ref{eq:NF}) hold
  true where $K$ is replaced by $E$, then 
  (\ref{eq:V}) and (\ref{eq:NF}) hold true. Note that $A$ is a basic
  $K$-abelian variety if and only if it is a basic $E$-abelian
  variety (Remark~\ref{rem:32}) . 
  Replacing $K$ by its center, we may assume that $K$ is
  commutative.  

  By Theorem~\ref{11}, there are $K$-abelian varieties $(A_0,\iota_0)$
  and $(A_0',\iota_0')$ over $\Fpbar$ such that
  $(A_0,\iota_0)\otimes_{\Fpbar} k \sim (A,\iota)$ and 
  $(A'_0,\iota'_0)\otimes_{\Fpbar} k \sim (A',\iota')$. We have
  a natural isomorphism 
\[ \Hom_{K}(A_0,A'_0)\otimes_{\Z} \Q\simeq 
  \Hom_{K}(A,A')\otimes_{\Z} \Q, \]
and natural identifications $V_\ell(A_0)=V_\ell(A)$ 
and $V_\ell(A'_0)=V_\ell(A')$ for $\ell \neq p$. For $\ell=p$, we have
also 
the identification $\Hom_K((N_0,
\calF),(N_0',\calF))=\Hom_K((N,\calF), (N',\calF))$, where
  $N_0$ and $N'_0$ are the isocrystals 
  associated to $(A_0,\iota_0)$ and $(A'_0,\iota'_0)$, respectively.
  Therefore, we are reduced to prove the
  statement where $k=\Fpbar$, which is done by Rapoport-Zink 
  (see \cite[Corollary 6.29, p.~293]{rapoport-zink}). \qed  
\end{proof}

\begin{cor}\label{37}
  Let $(A,\iota)$ and $(A',\iota')$ be two basic $B$-abelian varieties
  over an \ac field $k$ of \ch $p$, with slopes
  $\{\lambda_\bfp\}_{\bfp|p}$ and $\{\lambda'_\bfp\}_{\bfp|p}$,
respectively. Then $(A,\iota)$ and $(A',\iota')$ are  
isogenous if and only if  $\lambda_\bfp=\lambda'_\bfp$ and $\rank
N_\bfp=\rank N'_{\bfp}$ for all $\bfp|p$.
\end{cor}

\begin{proof}
  This follows from Proposition~\ref{36}. \qed
\end{proof}
\section{A correspondence}
\label{sec:04}

\subsection{}
\label{sec:41}

Let $\Fq$ be the finite field of $q=p^s$ elements. 
Let $\calA_{\Fq}$ denote the category of abelian varieties over $\Fq$
up to isogeny. Let $\calB$ be the category defined as follows, which
we call the category of basic abelian varieties with
endomorphisms over $\Fpbar$ up to isogeny. The
objects of $\calB$ consist of all triples $(F, A,\iota)$, where 
\begin{itemize}
\item $F$ is
a finite-dimensional commutative semi-simple $\Q$-algebra that admits a
positive involution, and
\item $(A,\iota)$ is a basic $F$-abelian variety over $\Fpbar$.
\end{itemize}
 
For any two
objects $\ul A_1=(F_1, A_1,\iota_1)$ and $\ul A_2=(F_2,A_2,\iota_2)$
in $\calB$, a morphism in $\Hom_{\calB}(\ul A_1, \ul A_2)$ 
is a pair $(\varphi,\wt \varphi)$, where 
\begin{itemize}
\item $\wt \varphi: F_1 \to F_2$ is
a $\Q$-linear algebra homomorphism in a broader sense that the image $\wt
\varphi(1_{F_1})$ of the identity $1_{F_1}$ may not be the identity
$1_{F_2}$, and 
\item $\varphi$ is an element in $\Hom(A_1,A_2)\otimes \Q$ 
which is $(F_1,F_2)$-equivariant in the sense that $\varphi \circ
\iota_1(a)= 
\iota_2(\wt \varphi (a))\circ \varphi$ for all $a\in F_1$. 
\end{itemize}

Note that if the map $\wt \varphi:F_1\to
F_2$ as above is surjective, 
then $\wt \varphi(1_{F_1})=1_{F_2}$ (as $\wt \varphi(1_{F_1})y=y$ for
all $y\in F_2$), i.e. it is also a
ring homomorphism. A reason we need to allow more general maps 
$\wt \varphi$ is as follows. Let $(A_i,\iota_i)$ be an $F_i$-abelian
variety for $i=1,2$, and $\iota_1\times \iota_2:F_1\times F_2\to
\End(A_1\times A_2)$ the product map. Then the map
$\varphi=\id_{A_1}\times 0:A_1\to A_1\times A_2$ is an
$(F_1,F_1\times F_2)$-equivariant with respect to the  map $\wt
\varphi=\id_{F_1}\times 0:F_1\to F_1\times F_2$. The latter map is not
a ring homomorphism.  

Clearly two objects $\ul A_1$ and $\ul A_2$ in $\calB$
are isomorphic if and only if there is a
$\Q$-algebra isomorphism $\wt \varphi: F_1\simeq F_2$, and an
$(F_1,F_2)$-equivariant quasi-isogeny $\varphi:A_1\to A_2$ over
$\Fpbar$. 

The category $\calB$ is not yet good enough  
in comparison with the category 
of abelian varieties with fixed endomorphism structures; there are
simply too many morphisms $\wt \varphi$ among the fields $F$. 
For example, when $F_1=F_2=F$, the usual notion of morphisms
between two $F$-abelian varieties would require $\wt \varphi$
to be the identity and not an arbitrary automorphism as 
in the category $\calB$. 

We introduce another category
$\calB^{\rm rig}$,
which we call the category of 
basic abelian varieties with rigidified endomorphisms over
$\Fpbar$ up to isogeny. The objects of $\calB^{\rm rig}$ consist of
all tuples
$(F,x, A,\iota)$ over $\Fpbar$, where $(F,A,\iota)$ is an object in
$\calB$ and $x\in F$ is an element generating $F$ over
$\Q$. Suppose $(F,x,A,\iota)$ is an object in $\calB^{rig}$, 
let $\Q[t]\to F$ be the
natural surjective map 
sending $t$ to $x$, and $f:\Q[t]\to \End^0(A)$ be the morphism
obtained by composing with the map $\iota$. 
Given two objects $\ul A_i=(F_i,x_i,
A_i,\iota_i)$ in $\calB^{\rm rig}$ ($i=1,2$), a morphism $\varphi:
\ul A_1\to \ul A_2$ in $\calB^{\rm rig}$ is an element $\varphi\in
\Hom(A_1, A_2)\otimes \Q$ such that $\varphi\circ f_1(a)=f_2(a)\circ
\varphi$ for all $a\in \Q[t]$, where $f_i:\Q[t]\to \End^0(A_i)$ are the
maps associated as above.     
In the case when $F_1=F_2=F$, we have 
\[ \Hom_F((A_1,\iota_1),(A_2,\iota_2))\otimes_{\Z} \Q =\Hom_{\calB^{\rm
    rig}}((F,x,A_1,\iota_1), (F,x, A_2,\iota_2)) \]   
for any element $x$ generates $F$ over $\Q$, which recovers the usual
notion of morphisms of $F$-abelian varieties (though we may not really
want the additional structure $x$). 

We shall embed $\calA_{\Fq}$ as a full subcategory of $\calB^{\rm
  rig}$. As the first step, we prove the following result.

\begin{prop}\label{41}
  Let $A$ be an abelian variety over $\Fq$ and $\pi_A$ its 
  Frobenius endomorphism. Put $F:=\Q(\pi_A)$ and $\iota:F\to \End^0(A)$
  for the inclusion. Then the $F$-abelian variety $(A,\iota)$ is basic. 
\end{prop}
\begin{proof}
Suppose that the finite field $k$ has $q=p^s$ elements. 
Let $A\sim \Pi_{i=1}^t A_i^{n_i}$ be the decomposition into
components up to isogeny, where each abelian variety $A_i$ 
is simple and $A_i\not\sim
A_j$ for any $i\neq j$. 
Let $\pi_i$ be the Frobenius endomorphism
of $A_i$ and put $F_i:=\Q(\pi_i)$. Then we have $F=\prod_i^t F_i$. Let
$\Sigma_{p,i}$ be the set of the primes $\bfp$ of $F_i$ over $p$. Thus,
$\Sigma_p$ is the disjoint union of $\Sigma_{p,i}$ for
$i=1,\dots, t$. Let $N$ (resp. $N_i$) 
be the isocrystal associated to the 
$F$-abelian
variety $\ul A=(A,\iota)$ (resp. $\ul A_i=(A_i,\iota_i)$). 
Clearly if $\bfp\in
\Sigma_{p,i}$ then $N_\bfp=N_{i,\bfp}^{n_i}$. In particular, $N_\bfp$
is isotypic for all $\bfp\in \Sigma_p$ if and only if $N_{i,\bfp}$ is
isotypic for all $i$ and all $\bfp\in \Sigma_{p,i}$. It follows from
Lemma~\ref{31} that $\ul A$ is basic if and only if $\ul A_i$ is basic
for all $i=1,\dots, t$. Therefore, it suffices to prove the statement
when $A$ is simple. In this case, as $F^s=\pi$ and $\pi\in F_\bfp$,
the component $N_\bfp$ has slope $\ord_\bfp(\pi)/s$. \qed
\end{proof}   

By Lemma~\ref{31}, if $K$ is any commutative semi-simple
$\Q$-subalgebra of the endomorphism algebra $\End^0(A)$ which
is stable under a Rosati involution and contains $F$, then $(A,i)$
with $i:K \subset \End^0(A)$, is also a basic $K$-abelian variety. 
Our way of making $A$ into a basic abelian
variety with endomorphism structures as in Proposition~\ref{41} is, 
after a suitable base change, the most
``economical'' one. Namely, one uses the least endomorphisms.  


\begin{prop}\label{4.15}
  Let $A$ be an abelian variety over $\Fq$ such that $\End(A)=\End(\ol
  A)$, where $\ol A=A\otimes \Fpbar$. 
  Suppose that $K$ is a commutative semi-simple $\Q$-algebra admitting
  a positive involution, and $(A,
  \iota)$ is a basic $K$-abelian variety. 
  Then $\iota(K)$ contains the center $F$ of the endomorphism
  algebra $\End^0(A)$. 
\end{prop}
\begin{proof}
  Let $\pi$ be the Frobenius endomorphism of $A$. Then for
  any positive integer $n$ one has $F=\Q(\pi^n)$ as $F$ is the center
  of the endomorphism algebra $\End^0(A\otimes_{\Fq} \F_{q^n})$. Now
  using Lemma~\ref{33}, there is a positive integer $n$ such that
  $\pi^n$ is contained in $\iota(K)$. As a result, the center $F$ is
  contained in $\iota(K)$. \qed  
\end{proof}

Now we define a functor $\Phi:\calA_{\Fq}\to \calB^{\rm rig}$ as follows.  
To each abelian variety $A$ over $\Fq$ we associate a tuple
$(F,\pi_A, \ol A,\iota)$, 
where $\pi_A$ is the Frobenius
endomorphism of $A$, $F:=\Q(\pi_A)$, $\ol A:=A\otimes_{\Fq} \Fpbar$ 
and $\iota:F\to \End^0(\ol A)$ 
is the inclusion. Clearly we have the associated map
\begin{equation}
  \label{eq:41}
  \Phi_*:\Hom(A_1,A_2)\otimes \Q \to \Hom_{\calB^{\rm rig}}
(\Phi(A_1),\Phi(A_2))
\end{equation}
as $\varphi\circ \iota_1(\pi_{A_1})= \iota_2(\pi_{A_2}) \circ \varphi$
for any map $\varphi\in\Hom(A_1,A_2)\otimes \Q$.   
   
\begin{thm}\label{42}
  The functor $\Phi:\calA_{\Fq} \to \calB^{\rm rig}$ is fully faithful.
\end{thm}
\begin{proof}
  Let $A_1$ and $A_2$ be two abelian varieties over $\Fq$, and let
  $\ul A_i:=(F_i, \pi_i, \ol A_i,\iota_i)$ be the associated object in
  $\calB^{\rm rig}$ for $i=1,2$. We must show that the associated map
  $\Phi_ *$ in (\ref{eq:41}) is bijective. It is clear that $\Phi_*$
  is injective. Let $\ol f: \ol A_1 \to \ol A_2 $ be an element in
  $\Hom_{\calB^{\rm rig}} (\Phi(A_1),\Phi(A_2))$, particularly 
  $\pi_2 \ol f= \ol f \pi_1$. As $\sigma_q(\ol f)=\pi_2 \ol f
  \pi_1^{-1}=\ol f$, where $\sigma_q\in \Gal(\Fpbar/\Fq)$ is the Frobenius
  map, the morphism $\ol f$ is defined over $\Fq$. \qed  
\end{proof}

\subsection{}
\label{sec:42}
We restrict the functor $\Phi$ to the objects 
for which the endomorphism algebras have a common center. 
Fix any abelian variety $A_0$ over $\Fq$. Let $\pi_0$
be the Frobenius endomorphism of $A_0$
over $\Fq$, $p(t)\in \Z[t]$ its minimal polynomial over $\Q$ and 
$F:=\Q[t]/(p(t))$. A commutative semi-simple
$\Q$-algebra $F$ arising in this way is called 
a {\it $q$-Weil $\Q$-algebra}. 

Let $\calA_{\pi_0, \Fq}$ denote the full subcategory of $\calA_{\Fq}$
consisting of all abelian
varieties $A$ such that the minimal
polynomial of the Frobenius endomorphism of $A$ 
is equal to $p(t)$. 
In other words, every abelian variety $A$ over $\Fq$ in
$\calA_{\pi_0,\Fq}$ shares the same simple components of $A_0$ up to 
isogeny. 
  
Let $\calB_F$ denote the category of basic $F$-abelian varieties over
$\Fpbar$ up to isogeny. Similarly we define a functor
\begin{equation}
  \label{eq:42}
  \Phi_F: \calA_{\pi_0,\Fq} \to \calB_F, \quad A\mapsto (\ol A, \iota), 
\end{equation}
where $\ol A:=A\otimes_{\Fq} \Fpbar$ and $\iota: F\to \End^0(\ol A)$
is the ring monomorphism sending $t$ to $\pi_A$. 
By Theorem~\ref{42}, we obtain the following result.

\begin{prop}\label{43}
  For any $q$-Weil $\Q$-algebra $F=\Q(\pi_0)$, the functor $\Phi_F:
  \calA_{\pi_0,\Fq}\to \calB_F$ is fully faithful. 
\end{prop}

\begin{remark}\label{rem:4.5}
The functor $\Phi_F$ is usually not essentially surjective. 
For example take $q=p^2$ and $\pi_0=p \zeta_6$ with $p\equiv 1
\pmod 3$. The corresponding
abelian variety $A_0$ is a simple supersingular abelian surface, and any
object in $\calA_{\pi_0,\Fq}$ is isogenous to a finite product of
copies of $A_0$. However, as $F=\Q(\sqrt{-3})$ and $p$ splits in $F$,
there is an ordinary elliptic curve $E$ over $\Fpbar$ and there is
an isomorphism $i: F\simeq \End^0(E)$. The $F$-elliptic curve $(E,i)$
is clearly in $\calB_F$ but is not isogenous to a finite product of
copies of $A_0$. In this case the functor $\Phi_F$ is not essentially
surjective. A point is that different Weil numbers can generate the
same field.   

\end{remark}

\section{A mass formula}
\label{sec:05}

\subsection{Within a simple isogeny class}
\label{sec:51}
\def\Isog{{\rm Isog}}
Let $\pi$ be a Weil $q$-number, $F=\Q(\pi)$ the number field
generated by $\pi$ over $\Q$, and $O_F$ the ring of integers in $F$. 
Let $\Isog(\pi)$ denote the simple isogeny class corresponding to
$\pi$ by the Honda-Tate theory \cite{tate:ht}. 
Let $A_0$ be an abelian variety over $\Fq$ in $\Isog(\pi)$ and 
put $d:=\dim(A_0)$.

Let $\Lambda(\pi)$ denote the set of isomorphism classes of abelian
varieties over $\Fq$ in $\Isog(\pi)$, and 
$\Lambda(\pi)^{\rm max}\subset \Lambda(\pi)$ be the subset consisting of
all abelian varieties $A$ such that the ring $O_F$ is contained in
$\End(A)$. 
Let
$\bfB_{d,O_F}$ denote the set of isomorphism classes of
$d$-dimensional basic $O_F$-abelian varieties over $\Fpbar$. 

The following lemma follows from Proposition~\ref{43}.

\begin{lemma}\label{51}
  The association $A\mapsto (\ol A, \iota)$ induces an injective map 
  $\Phi_{\pi}: \Lambda(\pi)^{\rm max} \to \bfB_{d,O_F}$.
\end{lemma}
\def\Mass{{\rm Mass}}

If $A\in \Lambda(\pi)^{\rm max}$ is an abelian variety over $\Fq$ 
and $(\ol A,\iota)$ the corresponding basic $O_F$-abelian variety over
$\Fpbar$, then clearly any $O_F$-linear polarization $\ol \lambda$ on $(\ol
A,\iota)$ descends uniquely to a polarization $\lambda$ on $A$ over 
$\Fq$. Particularly, the map $\lambda\mapsto \bar \lambda$ gives rise
to a one-to-one correspondence between polarizations on $A$ and
$O_F$-linear polarizations on $(\ol A, \iota)$ over $\Fpbar$. 
It follows that $A$ admits a principal polarization if and
only if $(\ol A, \iota)$ admits a principal $O_F$-linear
polarization. Moreover, we also have a natural 
isomorphism of finite groups 
\begin{equation}
  \label{eq:51}
   \Aut(A,\lambda)\simeq \Aut(\ol A, \ol \lambda, \iota).  
\end{equation}

Now we let $\Lambda(\pi)_1^{\rm max}$ be the set of isomorphism classes of
principally polarized abelian varieties $(A,\lambda)$ over $\Fq$ 
such that the underlying abelian variety $A$ belongs to 
$\Lambda(\pi)^{\rm max}$. The set $\Lambda(\pi)_1^{\rm max}$ could be
empty; nevertheless, it is always finite. This follows from the
finiteness of the set $\calA_{d,1}(\Fq)$ of $\Fq$-rational points of
the Siegel modular variety $\calA_{d,1}$,

Let $\calA_{d,O_F,1}$ be the moduli space over $\Fpbar$ of 
$d$-dimensional principally polarized $O_F$-abelian varieties, and 
$\bfB_{d,O_F,1}\subset \calA_{d,O_F,1}(\Fpbar)$ be its basic locus.
Then the map $\Phi_\pi$ induces an
injective map 
\begin{equation}
  \label{eq:52}
  \Phi_\pi: \Lambda(\pi)^{\rm max}_1 \to \bfB_{d,O_F,1}.
\end{equation}
We have the following commutative diagram
\[ 
\begin{CD}
  \Lambda(\pi)^{\rm max}_1 @>{\Phi_{\pi}}>> \bfB_{d,O_F,1} \\
  @VVV @VVV \\
  \Lambda(\pi)^{\rm max} @>{\Phi_{\pi}}>> \bfB_{d,O_F}, \\ 
\end{CD}\]
where the vertical maps forget the polarization.

The mass 
of $\Lambda(\pi)^{\rm max}_1$ is defined as
\begin{equation}
  \label{eq:53}
  \Mass(\Lambda(\pi)^{\rm max}_1):=\sum_{(A,\lambda)\in
  \Lambda(\pi)^{\rm max}_1} |\Aut(A,\lambda)|^{-1}
\end{equation}
if it is nonempty, and to be zero otherwise. 
Similarly, any finite subset $S\subset \calA_{d,O_F,1}(\Fpbar)$, 
the mass of
$S$ is defined as
\begin{equation}
  \label{eq:54}
  \Mass(S):=\sum_{(\ol A, \ol \lambda,\iota)\in S} |\Aut(\ol A,\ol
\lambda,\iota)|^{-1}
\end{equation}
if $S$ is nonempty and $\Mass(S)=0$ otherwise. 
It follows from (\ref{eq:51}) that 
\begin{equation}
  \label{eq:55}
  \Mass(\Lambda(\pi)^{\rm max}_1)=\Mass({\rm Im} \Phi_\pi).
\end{equation}

\subsection{An example with $\pi=\sqrt{p}$.}
\label{sec:52}

We consider a special case of the previous construction 
when $\pi=\sqrt{p}$. The result we obtain is the following.

\begin{thm}\label{52}
  Let $\pi=\sqrt{p}$. Then 
  the finite set $\Lambda(\pi)_1^{\rm max}$ is
  nonempty and we have
  \begin{equation}
    \label{eq:56}
    \Mass(\Lambda(\pi)_1^{\rm max}=\frac{1}{4}
    \, \zeta_{\Q(\sqrt{p})}(-1).  
  \end{equation}
\end{thm}

We need a general result. 


\begin{prop}\label{53}
  Let $F$ be a totally real field, $\calO:=O_F\otimes_{\Z} \Zp$ and 
$k$ an \ac field of \ch $p>0$.

{\rm (1)} Let $\ul M=(M,\<\,, \>, \iota_M)$ be a 
    supersingular separably
    quasi-polarized \dieu $\calO$-module
    over $k$ satisfying the following condition
\[ (*)\quad \tr(\iota_M(a))\cdot 
[F:\Q]=(\rank_W M)\cdot \tr_{F/\Q}(a), \quad
\forall\, a\in O_F. \] 
   Then there is a supersingular 
    principally polarized $O_F$-abelian
    variety $\ul A=(A,\lambda,\iota)$ 
    over $k$ whose \dieu module $M(\ul A)$ is isomorphic to $\ul M$.


{\rm (2)} Assume that $p$ is totally ramified in $F$.
    Then for any supersingular \dieu $\calO$-module $\ul
    M=(M,\iota_M)$  over
     $k$ of $W$-rank $2[F:\Q]$, 
    there is a principally polarized $O_F$-abelian variety 
    $\ul A=(A,\lambda,\iota)$ over $k$ such that the
    \dieu $\calO$-module $M(A,\iota)$ 
    is isomorphic to $\ul M$.  
\end{prop}
\begin{proof}
  (1) By \cite[Theorem 1.1]{yu:c}, there is a (prime-to-$p$ degree) 
   polarized $O_F$-abelian
   variety $\ul A=(A,\lambda,\iota)$ such that $M(\ul A)\simeq
   \ul M$. 
   We can choose a self-dual $(O_F\otimes \Z_\ell)$-lattice $L_\ell$ in
   $V_\ell(A)$  with respect to $e_{\lambda,\ell}$ for each prime
   $\ell\neq p$ with
   $L_\ell=T_\ell(A)$ for almost all $\ell$. The proof of 
   the existence of such a lattice $L_\ell$ 
   is elementary and left to the reader. 
   Then there is an $O_F$-abelian variety $(A',\iota')$ and a
   prime-to-$p$ degree $O_F$-linear quasi-isogeny 
   $\varphi': (A',\iota')\to (A,\iota)$ 
   such that $\varphi'_*(T_\ell(A'))=
   L_\ell$ for all $\ell\neq p$. 
   Then the pull-back $\lambda':=\varphi^* \lambda$ by $\varphi$ 
   is a principal polarization as $L_\ell$ is self-dual. 
   The object $(A',\lambda',\iota')$ is a desired one.  

  (2) Since there is only one prime of $O_F$ over $p$, the condition
      ($*$) is satisfied. 
      By \cite[Proposition 2.8]{yu:reduction}, the \dieu $\calO$-module 
      $\ul M$ admits a separable $\calO$-linear quasi-polarization,
      noting that    
      an equivalent condition (5) of loc. cit. 
      is satisfied when $p$ is
      totally ramified. Then the statement follows from (1). \qed
\end{proof}



Now we return to our case $F=\Q(\sqrt{p})$, where $\calO=O_F\otimes
\Zp=\Zp[\sqrt{p}]$. 
The prime $p$ is ramified in $F$ with
ramification index $e=2$. 
Clearly any member $A$ in 
$\Lambda(\pi)^{\rm max}$ is a superspecial abelian surface over
$\Fp$. The \dieu module $M=M(A)$ of $A$ is 
a rank 4 free $\Zp$-module together with a $\Z_p$-linear action by 
$O_F$. Therefore, $M\simeq \calO^2$ on which both the Frobenius
$\calF$ and the Verschiebung 
$\calV$ operate by $\sqrt{p}$. 
From this the Lie algebra $\Lie(A)=M/\calV M$ of $A$ is isomorphic to 
$\Fp\oplus \Fp$ as an $(O_F/p)$-module. 
In other words, $A$ has Lie type $(1,1)$ in the terminology
of \cite[Section 1]{yu:reduction}. Therefore, the injective map 
$\Phi_\pi: \Lambda(\pi)^{\rm max}\to \bfB_{2,O_F}$ factors through the
subset $\bfS\subset \bfB_{2,O_F}$ of superspecial abelian
$O_F$-surfaces of Lie type $(1,1)$. 

  
We first claim that the induced map 
\begin{equation}
  \label{eq:58}
  \Phi_\pi:\Lambda(\pi)^{\rm max}\to \bfS
\end{equation}
is bijective. Fix a member $A_0\in \Lambda(\pi)^{\rm max}$. 
By Waterhouse \cite[Theorem 6.2]{waterhouse:thesis}, there is a natural 
bijection between the set $\Lambda(\pi)^{\rm max}$ and the 
set $\Cl(\End(A_0))$ of right ideal classes. Since the map $\Phi_\pi$
is injective, it suffices to show that $\bfS$ has the same cardinality
as $\Cl(\End(A_0))$. 
Note that the isomorphism classes of (unpolarized) 
{\it superspecial} \dieu $\calO$-modules are uniquely 
determined by their Lie types \cite[Lemma 3.1]{yu:mass_hb}. 
It follows that the \dieu modules and Tate modules of 
any two members in $\bfS$ are mutually isomorphic (compatible 
with the actions of $O_F$). 
By (the unpolarized variant of) \cite[Theorem 2.1]{yu:mass_hb}, 
there is a natural bijection 
$\bfS\simeq \Cl(\End_{O_F}(\ol A_0))$. 
Since we have $\End(A_0)=\End_{O_F}(\ol A_0)$, our claim is
proved.     
 
Let $\bfS_1\subset \bfB_{2,O_F,1}$ be the subset consisting 
of objects $(A,\lambda,\iota)$ so that the underlying 
abelian $O_F$-surface $(A,\iota)$ belongs to
$\bfS$. Proposition~\ref{53} implies that $\bfS_1$ is
nonempty. 
Consider the
commutative diagram
\begin{equation}
  \label{eq:59}
  \begin{CD} 
   \Lambda(\pi)^{\rm max}_1 @>\Phi_{\pi}>> \bfS_1 \\
   @V{f_\Lambda}VV @V{f_\bfS}VV \\
   \Lambda(\pi)^{\rm max} @>\Phi_{\pi}>{\simeq}> \bfS \\ 
  \end{CD}
\end{equation}
Note that a member $A$ in $\Lambda(\pi)^{\rm max}$ admits a principal
polarization if and only if $\Phi_\pi(A)=(\ol A, \iota)$ admits a
principal $O_F$-linear polarization. Moreover, the equivalence classes
of principal polarizations on $A$ are in bijection with  
the equivalence classes
of principal $O_F$-linear polarizations on $(\ol A, \iota)$. It
follows that
the diagram (\ref{eq:59}) is cartesian, which particularly implies that
the map $\Phi_\pi: \Lambda(\pi)^{\rm max}_1\simeq \bfS_1$ is an
isomorphism. Thus, we have proved $\Mass(\Lambda(\pi)^{\rm
  max}_1)=\Mass(\bfS_1)$. 

Now we use the mass formula for $\Mass(\bfS_1)$ 
\cite[Theorem 3.7]{yu:mass_hb}\footnote{There is an error in the
  computation of the mass formula there. The error occurs in Lemma
  3.4 of loc. cit., where the unramified quadratic order $O_{\bfF_\grp'}$
  of $O_{\bfF_\grp}$ cannot be 
  written as $O_{\bfF_\grp}[\sqrt{c}]$ when $p=2$ as stated. As a result,
  the order $A_\epsilon$ when $\epsilon=0$ as in Lemma 3.4
  should be maximal, and the term $o_\grp$ should be always one in 
  that paper, particularly in the formulas of Theorems 4.4 and 4.5.} 
\begin{equation}
  \label{eq:5.10}
  \Mass(\Lambda(\pi)^{\rm max}_1)=\Mass(\bfS_1)=\frac{1}{4}
  \zeta_F(-1); 
\end{equation}
this proves Theorem~\ref{52}.

\subsection{Fibers of the map $f_{\bfS}$}
\label{sec:53}

We describe the fibers of the map $f_{\bfS}$ in (\ref{eq:59}). 
Suppose
$(A,\lambda_0,\iota)$ is a member in $\bfS_1$. Put
$D:=\End^0_{O_F}(A)$ and $O_D:=\End_{O_F}(A)$. Then $D$ is the
quaternion $F$-algebra ramified only at the two real places of $F$ and
$O_D$ is a maximal order. Note that the canonical involution $'$ 
is the unique positive involution on $D$. Therefore the Rosati
involution induced by any $O_F$-linear polarization must be $'$. 
Suppose $\lambda$ is another $O_F$-linear principal polarization, then
$\lambda=\lambda_0 a$ for some totally positive  
symmetric element $a\in O_D^\times$,   
so $a\in O_{F,+}^\times$, 
the set of totally positive units in
$O_F$. Suppose $b\in \Aut_{O_F}(A)$ is an $O_F$-linear
automorphism. Then the pull-back 
\[ b^* (\lambda_0 a)=b^t \lambda_0 a  b=\lambda_0 \lambda_0^{-1} b^t \lambda_0
b a =\lambda_0 (b' b) a . \]
Therefore, the set of equivalence classes of principal $O_F$-linear
polarizations on $(A,\iota)$ is in bijection with the set
$O_{F,+}^\times /\Nr(O_D^\times)$, where $\Nr: O_D\to O_F$ is the reduced
norm. In other words, we obtain an isomorphism 
\begin{equation}
  \label{eq:5.11}
  f_{\bfS}^{-1}(A,\iota)\simeq O_{F,+}^\times /\Nr(O_D^\times).
\end{equation}
As $\Nr(O_D^\times)\supset (O_F^\times)^2$, 
the group $O_{F,+}^\times /\Nr(O_D^\times)$ is a
homomorphism image of $O_{F,+}^\times /(O_F^{\times})^2$. 
The latter group has
1 or 2 elements according as the fundamental unit $\epsilon$ 
of $F$ has norm $-1$ or not. Therefore, if $N(\epsilon)=-1$, then 
$f_{\bfS}^{-1}(A,\iota)$ has one element. Otherwise, the fiber 
$f_{\bfS}^{-1}(A,\iota)$ has at most two elements. \\


\section*{Acknowledgments}
Theorem~\ref{11} is a result of an old (unpublished) 
preprint \cite{yu:basic} which
was prepared during the author's stay in the Max-Planck-Institut f\"ur
Mathematik in Bonn. The author is grateful to Ching-Li Chai for the
idea of the proof of Theorem 1.1 and thanks the MPIM for its kind
hospitality and excellent conditions. The author thanks Jiangwei Xue
for his patient clarifications and helpful comments 
on a revised version of the manuscript. 
The author is partially supported by the grants 
MoST 100-2628-M-001-006-MY4 and 104-2115-M-001-001-MY3. He also thanks
the referee for careful reading and helpful comments which improve the
exposition of the manuscript.

\end{document}